\documentclass[10pt]{amsart}%
\usepackage{eurosym}
\usepackage{amsfonts}
\usepackage{amsmath}
\usepackage{amssymb}
\usepackage{graphicx}
\usepackage{color}%
\providecommand{\U}[1]{\protect\rule{.1in}{.1in}}
\providecommand{\U}[1]{\protect\rule{.1in}{.1in}}
\newtheorem{theorem}{Theorem}

\newtheorem{definition}[theorem]{Definition}

\newtheorem{lemma}[theorem]{Lemma}

\newtheorem{remark}[theorem]{Remark}

\numberwithin{equation}{section}
\numberwithin{theorem}{section}

\newcommand{\interior}[1]{  {\kern0pt#1}^{\mathrm{o}}}

\begin{document}
\title[The $\infty$-eigenvalue problem with a sign-changing weight]{{The $\infty$-eigenvalue problem with a sign-changing weight}}
\author[U. Kaufmann, J. D. Rossi and J. Terra]{Uriel Kaufmann, Julio D. Rossi and Joana Terra }
\address[J. D. Rossi]{\\
\noindent Depto. de Matem\'{a}tica, FCEyN, Universidad de Buenos Aires, Ciudad
Universitaria, Pab 1 (1428), Buenos Aires, Argentina \\
\noindent email: \texttt{jrossi@dm.uba.ar}}

\address
[U. Kaufmann and J. Terra]{
\noindent FaMAF, Universidad Nacional de Cordoba, (5000), Cordoba, Argentina. \\
\noindent email: \texttt{kaufmann@mate.uncor.edu, jterra@famaf.unc.edu.ar}}

\maketitle

\begin{abstract}
Let $\Omega\subset\mathbb{R}^{n}$ be a smooth bounded domain and $m\in
C(\overline{\Omega})$ be a sign-changing weight function. For $1<p<\infty$,
consider the eigenvalue problem%
\[
\left\{
\begin{array}
[c]{ll}%
-\Delta_{p}u=\lambda m(x)|u|^{p-2}u & \text{in }\Omega,\\
u=0 & \text{on }\partial\Omega,
\end{array}
\right.
\]
where $\Delta_{p}u$ is the usual $p$-Laplacian. Our purpose in this article is
to study the limit as $p\rightarrow\infty$ for the eigenvalues $\lambda
_{k,p}\left(  m\right)  $ of the aforementioned problem. In addition, we
describe the limit of some normalized associated eigenfunctions when $k=1$.
\end{abstract}

\section{Introduction}

Our main goal in this paper is to study the limit as $p\rightarrow\infty$ in
the eigenvalue problem for the $p$-Laplacian with a sign-changing weight.

Let $\Omega\subset\mathbb{R}^{n}$ be a smooth bounded domain, $1<p<\infty$,
and consider
\[
\Delta_{p}u:=\mbox{div}(|\nabla u|^{p-2}\nabla u)
\]
the usual $p$-Laplacian operator. Let $m\in C(\overline{\Omega})$ be a function (the weight) that
changes sign in $\Omega$. We set
\[
\Omega_{+}:=\overline{\{m>0\}},\quad\Omega_{-}:=\overline{\{m<0\}},\quad
\Omega_{0}:=\{m=0\}.
\]
Since we assume that $m$ changes sign we have that $\Omega_{+}\not =
\emptyset$ and $\Omega_{-}\not = \emptyset$. 

The eigenvalue problem associated with the $p$-Laplacian with a weight
function $m$ is given by
\begin{equation}
\left\{
\begin{array}
[c]{ll}
-\Delta_{p}u(x) =\lambda m(x)|u|^{p-2}u(x) & x \in \Omega,\\
u(x) =0 & x \in \partial\Omega.
\end{array}
\right.  \label{eq.p}%
\end{equation}
It is a well-known fact in the literature (cf. \cite{arias,cuesta, cuesta2} and
references therein) that the first (positive) eigenvalue can be characterized
variationally as follows:
\begin{equation}
\displaystyle\lambda_{1,p}:=\lambda_{1,p}(m)=\inf_{\mathcal{A}^{+}(m)}%
\int_{\Omega}|\nabla u|^{p}>0,\label{1er.p+}%
\end{equation}
where
$\mathcal{A}^{+}(m):=\left\{  u\in W_{0}^{1,p}(\Omega
):\int_{\Omega}m|u|^{p}=1\right\}  $.

In a similar way the first negative eigenvalue is given by
\[
\displaystyle\mu_{1,p}:=\mu_{1,p}(m)=\displaystyle-\lambda_{1,p}%
(-m)=-\inf_{\mathcal{A}^{-}(m)}\int_{\Omega}|\nabla u|^{p}<0,
\]
where
$
\mathcal{A}^{-}(m):=\left\{  u\in W_{0}^{1,p}(\Omega
):\int_{\Omega}m|u|^{p}=-1\right\}  $.

Regarding higher eigenvalues, it is also known that a sequence of positive
eigenvalues $\lambda_{k,p}\left(  m\right)  $ can be obtained by the
Ljusternik-Schnirelman theory. In fact, it holds that
\[
0<\lambda_{1,p}(m)<\lambda_{2,p}(m)\leq\lambda_{3,p}(m)\leq\text{ }
...\ \leq\lambda_{k,p}\left(  m\right)  \rightarrow\infty\quad\text{as
}k\rightarrow\infty,
\]
see e.g. \cite{arias, GP1} and references therein. Of course, the same ideas
also give the existence of a sequence of negative
eigenvalues $\mu_{k,p}\left(  m\right)  $,
\[
0>\mu_{1,p}(m)>\mu_{2,p}(m)\geq\mu_{3,p}(m)\geq\text{ }
...\ \geq\mu_{k,p}\left(  m\right)  \rightarrow - \infty\quad\text{as
}k\rightarrow\infty.
\]

Eigenvalue problems have received an increasing amount of attention along the
last decades by many authors, being studied mainly via variational methods. We
quote, among many others, \cite{brasco, Anane, BK2, BKJ, BF, cuesta, cuesta2,
FranLamb, FP, GP1, Ju-Li-05, JLM, KL, Lin, LL, Lindq90, NRSanAS, Sm, SW}. In some of these references
the limit as $p\to \infty$ of the eingenvalue problem associated to the classical case, $m\equiv1$, 
was considered. In particular, this limit as $p\to \infty$ was studied
in detail in \cite{JLM} (for the first eigenvalue)
and \cite{Ju-Li-05} (for higher eigenvalues), see also \cite{BK2} for an anisotropic version.
	In those papers it is proved that
	$$
		\lambda_{1,\infty} (1) :=
		\lim_{p\to +\infty}\left(\lambda_{1,p} (1)\right)^{1/p}=
 		\inf \left\{
 		\displaystyle \frac{ \displaystyle
 		\|\nabla v\|_{L^\infty(\Omega)}}
	    {\displaystyle \|v\|_{L^\infty (\Omega)} }\colon
	    v\in W^{1,\infty}_0 (\Omega), v\not\equiv0\right\}
		= \frac{1}{R},
	$$
	where $R$ is the largest possible radius of a ball contained
	in $\Omega$.
	In addition, they take the limit as
	$p\to \infty$ in the eigenfunctions of the $p$-Laplacian
	eigenvalue problems (see \cite{JLM}) and are viscosity solutions of
	the following eigenvalue problem
	(called the infinity eigenvalue problem in the literature and studied in \cite{Champion,Crasta,HSY,JLM,Yu})
	\begin{equation*}
		\begin{cases}
			\min \left\{|\nabla u|-\lambda_{1,\infty}(1) u,\,
			\Delta_{\infty} u \right\}=0 &\text{in }\Omega,\\
			u=0& \mbox{on } \partial \Omega.
		\end{cases}
	\end{equation*}
	The operator
	$\Delta_{\infty}$ that appears here
	is called the $\infty$-Laplacian and is given by
	$\Delta_\infty u := -\langle D^2u Du, Du \rangle.$

Our main first result for the weighted case gives a geometric characterization of the first $\infty
$-eigenvalue and establishes that it is associated to an eigenfunction that
satisfies a limiting variational problem, as well as a partial differential
equation, the later being satisfied in the viscosity sense. These results
generalize classical results for the $p$-Laplace eigenvalue problem without
the weight. It is interesting to emphasize that positive $\infty
$-eigenvalues \textit{only} take into account the geometry of the set where
the weight $m$ is positive.

\begin{theorem}
\label{teo1.intro}
The limit as $p\rightarrow\infty$ in the minimization problem
\eqref{1er.p+} is given by
\begin{equation}
\lambda_{1,\infty}(m):=\lim_{p\rightarrow\infty}\sqrt[p]{\lambda_{1,p}%
(m)}=\inf_{u\in W_{0}^{1,\infty}(\Omega)}\frac{\Vert\nabla u\Vert_{L^{\infty
}(\Omega)}}{\Vert u\Vert_{L^{\infty}(\Omega_{+})}}. \label{lam.inf}%
\end{equation}
Moreover, this value $\lambda_{1,\infty}(m)$ has a geometric
characterization:
\[
\lambda_{1,\infty}(m)=\frac{1}{R_{+}},\quad\text{where\quad}R_{+}:=\max
_{x\in\Omega_{+}}d(x,\partial\Omega),
\]
i.e., $R_{+}$ is the radius of the largest ball in $\Omega$ centered at a
point in $\Omega_{+}$.

Let $u_{p}$ be an eigenfunction associated with
$\lambda_{1,p}(m)$, that is, a minimizer to \eqref{1er.p+}, normalized by
$\int_{\Omega}m|u_{p}|^{p}=1$. Then, up to a subsequence,
\[
u_{p}\rightarrow u_{\infty},
\]
uniformly in $\overline{\Omega}$ and weakly in $W_{0}^{1,q}(\Omega)$ for every
$1<q<\infty$. Also, $u_{\infty}\in W_{0}^{1,\infty}(\Omega)$, it is a
minimizer of \eqref{lam.inf} and a viscosity solution to
\begin{equation}
\left\{
\begin{array}
[c]{ll}
-\Delta_{\infty}v=0 & \text{ in }\{mv=0\}^{o},\\
\min\{-\Delta_{\infty}v,|\nabla v|-\lambda_{1,\infty}v\}=0 & \text{ in
}\{mv>0\},\\
\max\{-\Delta_{\infty}v,-|\nabla v|-\lambda_{1,\infty}v\}=0 & \text{ in
}\{mv<0\}, \\
v=0 & \text{ on } \partial \Omega.
\end{array}
\right.  \label{ecuacion.infty.intro}
\end{equation}
\end{theorem}

Concerning higher eigenvalues, which will be properly defined in section \ref{sect-higher}, we have been able to establish an upper bound.
This bound is analogous to the one obtained in \cite{Ju-Li-05} for the unweighted case,
but again the balls need to be centered in the set $\Omega_{+}$.
We have the following result:

\begin{theorem}
\label{teo2.intro} Let $\lambda_{k,p}$ be the $k$-th eigenvalue of the
$p$-Laplacian problem, as defined in \eqref{higher-eigen}. Then we have that
\[
\lim_{p\rightarrow\infty}\left(  \lambda_{k,p}\right)  ^{1/p}\leq\frac
{1}{R_{k,+}},
\]
where
$$
R_{k,+}:=\sup_{r>0}\{\text{\small{there are} }k\text{ \small{disjoint balls of radius}
}r\text{ \small{in} }\Omega\text{ \small{centered at} }x_{1},\dots,x_{k}\in\Omega_{+}\}.
$$
\end{theorem}

For the case of the
second $\infty$-eigenvalue, $k=2$, we are also able to completely determine
$\lambda_{2,\infty}$ and give a geometric characterization similar to the classical one by \cite{Ju-Li-05}, once
again depending only on the set where $m$ is positive.

\begin{theorem}
\label{teo3.intro} Let $\lambda_{2,p}$ be the second eigenvalue of the
$p$-Laplacian problem, as defined in \eqref{higher-eigen}. We have that
\begin{equation}
\lambda_{2,\infty}:=\lim_{p\rightarrow\infty}\left(  \lambda_{2,p}\right)
^{1/p}=\frac{1}{R_{2,+}},\label{ineqk2}%
\end{equation}
where
\[
R_{2,+}:=\sup_{r>0}\{\text{there are two disjoint balls }B_{r}(x_{1}),B_{r}%
(x_{2})\subset\Omega\text{ with }x_{1},x_{2}\in\Omega_{+}\}.
\]
\end{theorem}

\begin{remark}{\rm
Although in the above theorems we focus on the first positive eigenvalue, we can
obtain analogous results for the first negative eigenvalue. It holds that
\[
\mu_{1,\infty}(m):=\lim_{p\rightarrow\infty}-\sqrt[p]{-\mu_{1,p}(m)}%
=-\inf_{u\in W_{0}^{1,\infty}(\Omega)}\frac{\Vert\nabla u\Vert_{L^{\infty
}(\Omega)}}{\Vert u\Vert_{L^{\infty}(\Omega_{-})}}.
\]
In this case, we have
\[
\mu_{1,\infty}(m)=-\inf_{u\in W_{0}^{1,\infty}(\Omega)}\frac{\Vert\nabla
u\Vert_{L^{\infty}(\Omega)}}{\Vert u\Vert_{L^{\infty}(\Omega_{-})}}=-\frac
{1}{R_{-}}\text{,}%
\]
where $R_{-}$ the radius of the largest ball included in $\Omega$ centered at
a point in $\Omega_{-}$, i.e., $R_{-}:=\max_{x\in\Omega_{-}}d(x,\partial
\Omega)$. Also, the limit of the associated eigenfunctions satisfies an eigenvalue problem
analogous to \eqref{ecuacion.infty.intro}.

A similar result concerning higher eigenvalues also holds for the negative ones.}
\end{remark}

Finally, let us observe that with the same ideas we can analyze a slightly different operator. Namely, we now
add a term $C(x)|u|^{p-2}u$ to the $p$-Laplacian and obtain the following eigenvalue problem:
\begin{equation}
\left\{
\begin{array}
[c]{ll}
-\Delta_{p}u(x)+C(x)|u(x)|^{p-2}u(x)=\lambda m(x)|u|^{p-2}u(x) & x\in\Omega,\\
u(x)=0 & x\in\partial\Omega,
\end{array}
\right.  \label{eq.p.otherintro}
\end{equation}
where $C$ is continuous and positive in $\overline{\Omega}$ and 
$m$ changes sign and satisfies the previous conditions.
For this problem, it is known (see \cite{cuesta2}) that there exists a principal eigenvalue given by
\begin{equation}
\lambda_{1,p}(C,m)=\min_{\int_{\Omega}m|u|^{p}=1}\int_{\Omega}|\nabla
u|^{p}+C|u|^{p}. \label{eigen.Cintro}
\end{equation}
Concerning the limit as $p\to \infty$ we have the following result.

\begin{theorem}
\label{teo.C.intro}
The limit as $p\rightarrow\infty$ in the minimization problem
\eqref{eigen.Cintro} is given by
\begin{equation} \label{lim.C.intro}
\lambda_{1,\infty}(C,m):=\lim_{p\rightarrow\infty}\sqrt[p]{\lambda_{1,p}
(C,m)} = \max\left\{ \frac{1}{R_{+}},1 \right\},
\end{equation}
where, as before, $R_{+}:=\max
_{x\in\Omega_{+}}d(x,\partial\Omega)$.
\end{theorem}

The rest of the paper is organized as follows: In Section \ref{sect-prelim} we
collect previous necessary results, namely we recall the definition of
viscosity solution and the equivalence between viscosity and weak solution in
the $p$-Laplacian setting. Next, in Section \ref{sect-1st}, we concentrate on
the first $\infty$-eigenvalue and prove Theorem \ref{teo1.intro}. In
Section \ref{sect-higher} we deal with higher eigenvalues and include some simple examples to see how the
eigenvalues depend on the set $\Omega_+$. Finally, in Section 
\ref{sect.Other} we deal with Theorem \ref{teo.C.intro}.

\section{Preliminary results}

\label{sect-prelim}

In this section we collect some results that will be used along this paper.

First, we observe that we can rewrite the first equation in
\eqref{eq.p} as
\begin{equation} \label{ec.diver}
-\mbox{div} (|\nabla u|^{p-2} \nabla u)  =\lambda m(x) |u|^{p-2}u
\end{equation}
or, expanding the divergence operator, as
\begin{equation}
-|\nabla u_{p}|^{p-4}\left(  |\nabla u_{p}|^{2}\Delta u_{p}+(p-2)\Delta
_{\infty}u_{p}\right)  =\lambda m(x) |u_{p}|^{p-2}u_{p}.\label{superecua}
\end{equation}
This equation is in divergence form and is nonlinear. Nevertheless it is elliptic (degenerate) and there are
multiple ways in which we can define solution to this problem. 
The first one is the concept of weak solution (that is closely related to the variational
nature of this problem).
\begin{definition}
A function $u \in W^{1,p}_0 (\Omega)$ is a
\textit{weak solution} of \eqref{ec.diver} if 
$$
\int_\Omega |\nabla u|^{p-2} \nabla u \nabla \varphi = \int_\Omega \lambda m(x) |u|^{p-2}u \varphi
$$
for every $\varphi \in W^{1,p}_0 (\Omega)$.
\end{definition}

Since our goal is to consider the limit as $p\rightarrow\infty$, we need to
choose an appropriate concept of solution such that it is somehow
\textquotedblleft stable\textquotedblright\ under the limit, in order to identify the limiting problem. 
The right
notion of solution to this problem is the viscosity one (see
e.g. \cite{JLM2}). Notice that the limit equation that appears in \eqref{ecuacion.infty.intro} 
is not in divergence form.

For the reader's convenience we briefly include the basics of the notion of
viscosity solution, that will be used in the next section to establish the
equation satisfied by the limiting function. Let $x,y\in\mathbb{R}$,
$z\in\mathbb{R}^{N}$, and $S$ be a real symmetric matrix. We define the
following continuous function
\[
H_{p}(x,y,z,S):=  \displaystyle-|z|^{p-4}\Big(|z|^{2}%
\mbox{trace}(S)+(p-2)\langle S\cdot z,z\rangle\Big)\displaystyle-\lambda
_{1,p}m(x)|u_{p}|^{p-2}u_{p}.
\]
Observe that $H_{p}$ is elliptic in the sense that $H_{p}(x,y,z,S)\geq
H_{p}(x,y,z,S^{\prime})$ if $S\leq S^{\prime}$ in the sense of bilinear forms,
and also that \eqref{superecua} can then be written as $H_{p}(x,u_{p},\nabla
u_{p},D^{2}u_{p})=0$. We are thus interested in viscosity sub and
supersolutions of the partial differential equation
\begin{equation}
\left\{
\begin{array}
[c]{ll}%
H_{p}(x,u,\nabla u,D^{2}u)=0 & \mbox{ in }\Omega,\\
u=0 & \mbox{ on }\partial\Omega.
\end{array}
\right.  \label{visco}%
\end{equation}

\begin{definition}
An upper semicontinuous function $u$ defined in $\Omega$ is a
\textit{viscosity subsolution} of $(\ref{visco})$ if $u|_{\partial\Omega}%
\leq0$ and, whenever $x_{0}\in\Omega$ and $\phi\in C^{2}(\Omega)$ are such that

\begin{itemize}
\item[i)] $u(x_{0})=\phi(x_{0})$,

\item[ii)] $u(x)<\phi(x)$ if $x\neq x_{0}$,
\end{itemize}

\noindent then
\[
H_{p}(x_{0},\phi(x_{0}),\nabla\phi(x_{0}),D^{2}\phi(x_{0}))\leq0.
\]
\end{definition}

\begin{definition}
A lower semicontinuous function $u$ defined in $\Omega$ is a \textit{viscosity
supersolution} of $(\ref{visco})$ if $u|_{\partial\Omega}\geq0$ and, whenever
$x_{0}\in\Omega$ and $\phi\in C^{2}(\Omega)$ are such that

\begin{itemize}
\item[i)] $u(x_{0})=\phi(x_{0})$,

\item[ii)] $u(x)>\phi(x)$ if $x\neq x_{0}$,
\end{itemize}

\noindent then
\[
H_{p}(x_{0},\phi(x_{0}),\nabla\phi(x_{0}),D^{2}\phi(x_{0}))\geq0.
\]
\end{definition}

We observe that in both of the above definitions the second condition is
required just in a neigbourhood of $x_{0}$ and the strict inequality can be
relaxed. We refer to \cite{CIL} for more details about the general theory of
viscosity solutions, and to \cite{JLM2} for viscosity solutions related to the
$\infty$-Laplacian and the $p$-Laplacian operators. The following result can
be shown as in \cite[Proposition 2.4]{MRU} (recall that $\lambda_{k,p}$ are as
in \eqref{higher-eigen} below).

\begin{lemma}
\label{sol.debil.es.viscosa} A continuous weak solution to the eigenvalue problem
\begin{equation}
\left\{
\begin{array}
[c]{ll}%
-\Delta_{p}u(x)=\lambda_{k,p}m(x)|u_{p}|^{p-2}u_{p} (x) & x \in \Omega,\\
u(x) =0 & x \in \partial\Omega,
\end{array}
\right.  \label{1.1.u}%
\end{equation}
is also a viscosity solution in the sense of the previous definition.
\end{lemma}

Note that, from the results in \cite{cuesta}, variational eigenvalues in the sequences of positive/negative eigenvalues to our problem have
associated eigenfunctions that are weak solutions (and hence viscosity solutions) to \eqref{1.1.u}.

\section{The first eigenvalue.}

\label{sect-1st}

\textit{Proof of Theorem \ref{teo1.intro}}. Let $u_{p}$ be a solution to
\eqref{eq.p} and $v$ be any test function. We have that, due to the
variational characterization,
\[
\lambda_{1,p}=\displaystyle\frac{\displaystyle\int_{\Omega}|\nabla u_{p}|^{p}%
}{\displaystyle\int_{\Omega}m|u_{p}|^{p}}\leq\frac{\displaystyle\int_{\Omega
}|\nabla v|^{p}}{\displaystyle\int_{\Omega}m|v|^{p}}.
\]
Let $B_{r}(c):=\{x\in\mathbb{R}^{n}:\left\vert {x-c}\right\vert <r\}$ be a
ball contained in $\Omega$ and centered at a point $c\in\{m>0\}$, and define
the following function:
\[
w:=\left\{
\begin{array}
[c]{ll}%
d(x,\partial B_{r}(c)) & \text{if }x\in B_{r}(c),\\
0 & \text{if }x\not \in B_{r}(c).
\end{array}
\right.
\]
Using $w$ as a test function above we have%
\[
\lambda_{1,p}\leq\frac{|\Omega|}{\displaystyle\int_{B_{r}(c)}m|w|^{p}},
\]
which is equivalent to
\begin{equation}
\lambda_{1,p}^{1/p}\leq\frac{|\Omega|^{1/p}}{\left\Vert {w}\right\Vert
_{L^{p}(B_{r}(c),m)}}. \label{unop}%
\end{equation}

Now, choosing $\delta>0$ such that $B_{\delta}(c)\subset\{m>0\}$, we observe
that
\[%
\begin{array}
[c]{l}
\displaystyle\left\Vert {w}\right\Vert _{L^{p}(B_{r}(c),m)}
\displaystyle=\left(  \int_{B_{\delta}(c)}m|w|^{p}+\int_{B_{r}(c)\setminus
B_{\delta}(c)}m|w|^{p}\right)  ^{1/p}\\
\displaystyle \qquad \geq\left(  \int_{B_{\delta/2}(c)}|m^{1/p}(r-\delta/2)|^{p}-\Vert
m\Vert_{L^{\infty}}\int_{B_{r}(c)\setminus B_{\delta}(c)}|r-\delta
|^{p}\right)  ^{1/p}\\
\displaystyle \qquad \geq(r-\delta/2)\left(  C-\Vert m\Vert_{L^{\infty}}\int
_{B_{r}(c)\setminus B_{\delta}(c)}\left(  \frac{r-\delta}{r-\delta/2}\right)
^{p}\right)  ^{1/p}\\
\qquad\rightarrow r-\delta/2\qquad\mbox{ as }p\rightarrow\infty.
\end{array}
\]
On the other hand,
\[
\begin{array}
[c]{l}
\displaystyle\left\Vert {w}\right\Vert _{L^{p}(B_{r}(c),m)}
\displaystyle=\left(  \int_{B_{\delta}(c)}m|w|^{p}+\int_{B_{r}(c)\setminus
B_{\delta}(c)}m|w|^{p}\right)  ^{1/p}\\
\displaystyle \qquad \leq\left(  \int_{B_{\delta/2}(c)}|m^{1/p}r|^{p}+\Vert
m\Vert_{L^{\infty}}\int_{B_{r}(c)\setminus B_{\delta}(c)}|r-\delta
|^{p}\right)  ^{1/p}\\
\qquad\rightarrow r\qquad\mbox{ as }p\rightarrow\infty.
\end{array}
\]
Since $\delta$ can be chosen arbitrarily small we conclude that $\lim
_{p\rightarrow\infty}\left\Vert {w}\right\Vert _{L^{p}(B_{r}(c),m)}=r$. Now,
taking limits in $p$ in (\ref{unop}) we deduce that
\begin{equation}
\limsup_{p\rightarrow\infty}\lambda_{1,p}^{1/p}\leq\frac{1}{r}.
\label{limlambda1}%
\end{equation}
Therefore, as this inequality holds being $r$ the radius of any ball contained
in $\Omega$ and centered at $c\in\{m>0\}$, we get
\begin{equation}
\limsup_{p\rightarrow\infty}\lambda_{1,p}^{1/p}\leq\inf_{\left\{
r>0:B_{r}(c)\subset\Omega,\text{ }c\in\{m>0\}\right\}  }\frac{1}{r}%
=\inf_{\left\{  r>0:B_{r}(c)\subset\Omega,\text{ }c\in\Omega_{+}\right\}
}\frac{1}{r}=\frac{1}{R_{+}}. \label{largo}%
\end{equation}

On the other hand, using H\"{o}lder's inequality we have, for $q<p$,
\begin{equation}
\left\Vert {\nabla u_{p}}\right\Vert _{q}\leq\left\Vert {\nabla u_{p}%
}\right\Vert _{p}|\Omega|^{1/q-1/p}=\lambda_{1,p}^{1/p}|\Omega|^{1/q-1/p}\leq
C. \label{upnorm}%
\end{equation}
Hence, $\{u_{p}\}$ is a bounded sequence in $W_{0}^{1,q}(\Omega)$ and therefore
there is a subsequence (that we still call $u_{p}$) that converges weakly in
$W_{0}^{1,q}(\Omega)$ and uniformly in $\overline{\Omega}$ to a limit
$u_{\infty}$ (we are using here that $W_{0}^{1,q}(\Omega)\hookrightarrow
C(\overline{\Omega})$ when $q>N$). By a diagonal procedure we can obtain a
subsequence $u_{p}$ that converges weakly in $W_{0}^{1,q}(\Omega)$ for every
$1<q<\infty$ and uniformly in $\overline{\Omega}$ to $u_{\infty}$.

Now, recalling (\ref{largo}) and letting $p\rightarrow\infty$ in
(\ref{upnorm}) we derive that%
\[
\left\Vert {\nabla u_{\infty}}\right\Vert _{q}\leq\limsup_{p\rightarrow\infty
}\lambda_{1,p}^{1/p}|\Omega|^{1/q-1/p}\leq\frac{1}{R_{+}}|\Omega|^{1/q},
\]
and now taking $q\rightarrow\infty$ we finally get
\[
\left\Vert {\nabla u_{\infty}}\right\Vert _{\infty}\leq\frac{1}{R_{+}}.
\]
Hence $u_{\infty}$ belongs to $W_{0}^{1,\infty}(\Omega)$. Moreover, since we normalized the eigenfunctions by $\int_{\Omega
}m|u_{p}|^{p}=1$,
\[
\displaystyle1=\left(  \int_{\Omega}m|u_{p}|^{p}\right)  ^{1/p}\leq\left(
\int_{\Omega}m^{+}|u_{p}|^{p}\right)  ^{1/p}\rightarrow\Vert u_{\infty}%
\Vert_{L^{\infty}(\Omega_{+})}\qquad\mbox{ as }p\rightarrow\infty.
\]
Therefore, $\Vert u_{\infty}\Vert_{L^{\infty}(\Omega_{+})}\geq1$. Next we
notice that (since $u_{\infty}$ is Lipschitz continuous in $\overline{\Omega}%
$) there exists $x_{0}\in\Omega_{+}$ with
\[
u_{\infty}(x_{0})=\Vert u_{\infty}\Vert_{L^{\infty}(\Omega_{+})}\geq1.
\]
Now we observe that, if we take $y\in\partial\Omega$ such that $|x_{0}-y|=$
dist$(x_{0},\partial\Omega)$, we have
\[
1\leq u_{\infty}(x_{0})=u_{\infty}(x_{0})-u_{\infty}(y)\leq\left\Vert {\nabla
u_{\infty}}\right\Vert _{\infty}|x_{0}-y|\leq\frac{1}{R_{+}}|x_{0}-y|.
\]
Hence, as $R_{+}:=\max_{x\in\Omega_{+}}d(x,\partial\Omega)$ we get that all
the previous inequalities must be equalities and so
\[
u_{\infty}(x_{0})=1,\qquad\left\Vert {\nabla u_{\infty}}\right\Vert _{\infty
}=\frac{1}{R_{+}},\qquad\mbox{and}\qquad d(x_{0},\partial\Omega)=R_{+}.
\]
Notice that this implies that $u_{\infty}$ is a minimizer for the limit
variational problem, that is,
\[
\frac{\left\Vert \nabla u_{\infty}\right\Vert _{L^{\infty}(\Omega)}%
}{\left\Vert u_{\infty}\right\Vert _{L^{\infty}(\Omega^{+})}}=\inf_{v\in
W_{0}^{1,\infty}(\Omega)}\frac{\left\Vert \nabla v\right\Vert _{L^{\infty
}(\Omega)}}{\left\Vert v\right\Vert _{L^{\infty}(\Omega^{+})}}.
\]

On the other hand, again employing \eqref{upnorm} we infer that
\[
\left\Vert {\nabla u_{\infty}}\right\Vert _{q}\leq\liminf_{p\rightarrow\infty
}\left\Vert {\nabla u_{p}}\right\Vert _{q}\leq\left(  \liminf_{p\rightarrow
\infty}\lambda_{1,p}^{1/p}\right)  |\Omega|^{1/q},
\]
and letting $q\rightarrow\infty$ we conclude that
\[
\frac{1}{R_{+}}=\left\Vert {\nabla u_{\infty}}\right\Vert _{\infty}\leq
\liminf_{p\rightarrow\infty}\lambda_{1,p}^{1/p}.
\]
Taking into account \eqref{largo} we derive that there exists the limit as
$p\rightarrow\infty$ of $(\lambda_{1,p})^{1/p}$ ($:=\lambda_{1,\infty}$) and
that is given by
\[
\lambda_{1,\infty}=\frac{1}{R_{+}}=\frac{\left\Vert \nabla u_{\infty
}\right\Vert _{L^{\infty}(\Omega)}}{\left\Vert u_{\infty}\right\Vert
_{L^{\infty}(\Omega^{+})}}=\inf_{v\in W_{0}^{1,\infty}(\Omega)}\frac
{\left\Vert \nabla v\right\Vert _{L^{\infty}(\Omega)}}{\left\Vert v\right\Vert
_{L^{\infty}(\Omega^{+})}}.
\]
This ends the proof of the first assertion of the theorem.

The next and final step in this proof is to find the equation satisfied by
$u_{\infty}$. We start by addressing the set $\{m=0\}^{o}$ and prove that
\[
-\Delta_{\infty}u_{\infty}=0\quad\text{in }\{m=0\}^{o}\text{ in the viscosity
sense.}%
\]
Following the definition of viscosity solution as stated in the previous
section, let $x_{0}\in\{m=0\}^{o}$ and $\phi\in C_{loc}^{2}$ be such that
$u_{\infty}(x_{0})=\phi(x_{0})$ and $u_{\infty}(x)<\phi(x)$, for all $x\in B$
where $B$ is an open ball containing $x_{0}$. We need to show that
\[
-\Delta_{\infty}\phi(x_{0})\leq0.
\]

Since $u_{p}\rightarrow u_{\infty}$ uniformly, the function $u_{p}-\phi$
reaches a maximum over $B$ at an interior point, say $x_{p}$. First we see
that $x_{0}$ is the only limit point of $\{x_{p}\}$. In fact, if there existed
another cluster point $x^{\ast}\neq x_{0}$ , then $x_{p^{\prime}}\rightarrow
x^{\ast}$ for $x_{p^{\prime}}$ maximum point of $u_{p^{\prime}}-\phi$ in $B$.
In particular, we would have%
\[
u_{p^{\prime}}(x_{p^{\prime}})-\phi(x_{p^{\prime}})\geq u_{p^{\prime}}%
(x_{0})-\phi(x_{0}).
\]
Letting $p^{\prime}$ tend to infinity and recalling that $u_{p}$ tends to
$u_{\infty}$ in $C(\Omega)$ due to classical compactness theorems,%

\[
u_{\infty}(x^{*})-\phi(x^{*})\geq u_{\infty}(x_{0})-\phi(x_{0}) = 0,
\]
which is a contradiction with the definition of $x_{0}$ and $\phi$. Therefore,
$x_{p^{\prime}}\rightarrow x_{0}$.

Since $x_{p^{\prime}}$ is a maximum point of $u_{p^{\prime}}-\phi$ in $B$ from
the equation satisfied by $u_{p}$ at $x_{p}\in B$ we obtain
\[
-|\nabla\phi(x_{p})|^{p-4}\left(  |\nabla\phi(x_{p})|^{2}\Delta\phi
(x_{p})+(p-2)\Delta_{\infty}\phi(x_{p})\right)  \leq0.
\]
Assuming $\phi$ is such that $\nabla\phi(x_{0}) \neq0 $ (otherwise we
trivially obtain $-\Delta_{\infty}\phi(x_{0}) = 0$) we have that $\nabla
\phi(x_{p}) \neq0 $, and hence we may divide by $(p-2)|\nabla\phi
(x_{p})|^{p-4}$ and obtain
\[
-\frac{|\nabla\phi(x_{p})|^{2}\Delta\phi(x_{p})}{p-2} - \Delta_{\infty}%
\phi(x_{p})\leq0.
\]

Now, letting $p\rightarrow\infty$ we obtain
\[
-\Delta_{\infty}\phi(x_{0})\leq0,
\]
that is, $u_{\infty}$ is a viscosity subsolution to $-\Delta_{\infty}v=0.$

Similarly one can establish that $u_{\infty}$ is a viscosity supersolution to
$-\Delta_{\infty}v=0,$ and hence we conclude that $u_{\infty}$ is a viscosity
solution to $-\Delta_{\infty}v=0$ in $\Omega_{0}$.

Now, we deal with the other cases. We start by looking at points where
$u_{\infty}$ is positive.

We consider $x_{0}\in\{m>0\}$ and $\phi\in C_{loc}^{2}$ be such that
$u_{\infty}(x_{0})=\phi(x_{0})$ and $u_{\infty}(x)<\phi(x)$, for all $x\in B$
where $B$ is an open ball containing $x_{0}$. Following the steps used before
we now arrive to
\[
-\frac{|\nabla\phi(x_{p})|^{2}\Delta\phi(x_{p})}{p-2}-\Delta_{\infty}%
\phi(x_{p})\geq\frac{\lambda_{1,p}m(x_{p})|u_{p}(x_{p})|^{p-1}}{(p-2)|\nabla
\phi(x_{p})|^{p-4}}.
\]
Again we may assume that $\phi$ is such that $\nabla\phi(x_{p})\neq0$ (since
the right hand side is positive) and then we may divide by $(p-2)|\nabla
\phi(x_{p})|^{p-4}$. Again due to the fact that the right hand side is positive we may rewrite it
as
\[
\frac{\lambda_{1,p}m(x_{p})|u_{p}(x_{p})|^{p-1}}{(p-2)|\nabla\phi
(x_{p})|^{p-4}}=\frac{1}{p-2}\left(  \frac{\lambda_{1,p}^{1/p}m^{1/p}%
(x_{p})|u_{p}(x_{p})|^{\frac{p-1}{p}}}{|\nabla\phi(x_{p})|^{\frac{p-4}{p}}%
}\right)  ^{p}.
\]
As $p\rightarrow\infty$ we have%
\[
-\Delta_{\infty}\phi(x_{0})\geq\lim_{p\rightarrow\infty}\frac{1}{p-2}\left(
\frac{\lambda_{1,p}^{1/p}m^{1/p}(x_{p})|u_{p}(x_{p})|^{\frac{p-1}{p}}}%
{|\nabla\phi(x_{p})|^{\frac{p-4}{p}}}\right)  ^{p}.
\]

Since $\phi$ is in $C^{2}$ the left hand side is well defined and that implies
that the right hand side must be finite. This in turn leads to
\[
\lambda_{1,\infty}\phi(x_{0})\leq|\nabla\phi(x_{0})|.
\]
Therefore, we have obtained
\[
\min\{-\Delta_{\infty}\phi(x_{0}), |\nabla\phi(x_{0})| - \lambda_{1,\infty}
\phi(x_{0}) \} \geq0.
\]
That is, $u_{\infty}$ is a viscosity subsolution.

To show that $u_{\infty}$ is a viscosity supersolution we consider $x_{0}%
\in\{m>0\}$ and $\phi\in C_{loc}^{2}$ be such that $u_{\infty}(x_{0}%
)=\phi(x_{0})$ and $u_{\infty}(x)>\phi(x)$, for all $x\in B$ where $B$ is an
open ball containing $x_{0}$. In this case we arrive to
\[
-\frac{|\nabla\phi(x_{p})|^{2}\Delta\phi(x_{p})}{p-2}-\Delta_{\infty}%
\phi(x_{p})\leq\frac{\lambda_{1,p}m(x_{p})|u_{p}(x_{p})|^{p-1}}{(p-2)|\nabla
\phi(x_{p})|^{p-4}}.
\]
Again we may assume that $\phi$ is such that $\nabla\phi(x_{p})\neq0$ and then
we may divide by $(p-2)|\nabla\phi(x_{p})|^{p-4}$. Since the right hand side
is positive we may rewrite it as
\[
\frac{\lambda_{1,p}m(x_{p})|u_{p}(x_{p})|^{p-1}}{(p-2)|\nabla\phi
(x_{p})|^{p-4}}=\frac{1}{p-2}\left(  \frac{\lambda_{1,p}^{1/p}m^{1/p}%
(x_{p})|u_{p}(x_{p})|^{\frac{p-1}{p}}}{|\nabla\phi(x_{p})|^{\frac{p-4}{p}}%
}\right)  ^{p}.
\]
As $p\rightarrow\infty$ we get%
\[
-\Delta_{\infty}\phi(x_{0})\leq\lim_{p\rightarrow\infty}\frac{1}{p-2}\left(
\frac{\lambda_{1,p}^{1/p}m^{1/p}(x_{p})|u_{p}(x_{p})|^{\frac{p-1}{p}}}%
{|\nabla\phi(x_{p})|^{\frac{p-4}{p}}}\right)  ^{p}%
\]

Now, if
\[
|\nabla\phi(x_{0})| - \lambda_{1,\infty} \phi(x_{0}) = |\nabla\phi(x_{0})| -
\lambda_{1,\infty} u (x_{0}) >0
\]
then the right hand side goes to 0 as $p\to\infty$ and we get that
\[
\lambda_{1,\infty}\phi(x_{0}) < |\nabla\phi(x_{0})| \implies-\Delta_{\infty
}\phi(x_{0}) \leq0
\]
Therefore, we have obtained
\[
\min\{-\Delta_{\infty}\phi(x_{0}), |\nabla\phi(x_{0})| - \lambda_{1,\infty}
\phi(x_{0}) \} \leq0.
\]
That is, $u_{\infty}$ is a viscosity supersolution.

The equation in the set $\{m<0\}$ when $u_{\infty}$ is positive can be
obtained with analogous computations. When $u_{\infty}$ is negative we argue
in the same way noticing that the inequalities are reversed. \qed

\section{Higher eigenvalues.}

\label{sect-higher}

In this section we analyse the case of higher eigenvalues. In order to do so,
we first recall that there exists a sequence of positive eigenvalues that can
be constructed by variational methods. Since $m^{+}:=\max\{m,0\}\not \equiv 0$, $m\in
C(\overline{\Omega})$ and $\Omega$ is a bounded domain, we are in the setting
described in \cite{cuesta}. If we want to allow the domain $\Omega$ to
be unbounded we would need other restrictions on $m$ to assure our variational
problem is set on a manifold (see \cite{Sm} and also 
\cite{SW} for further details) and similar results hold.

In fact, positive eigenvalues to our problem correspond (via Lagrange
multipliers type arguments) to positive critical values of the functional
\[
\Phi:W_{0}^{1,p}(\Omega)\rightarrow\mathbb{R},\qquad\Phi(u):=\int_{\Omega
}|\nabla u|^{p},
\]
restricted to the $C^{1}$ manifold 
$\mathcal{A}^{+}(m):=\left\{  u\in W_{0}^{1,p}(\Omega
):\int_{\Omega}m|u|^{p}=1\right\}  $. Such critical values can be characterized
by being the image through $\Phi$ of a function $u\in \mathcal{A}^{+}(m)$ such that
$\Phi^{\prime}(u)$ is orthogonal to the tangent space of $\mathcal{A}^{+}(m)$ at $u$,
$T_{u}\mathcal{A}^{+}(m)$. We emphasize that, since we choose $m^{+}\not \equiv 0$,
then $\mathcal{A}^{+}(m)\neq\emptyset$ is a $C^{1}$ manifold.

Therefore, we now focus on the analysis of such critical values. Since this is a
nonlinear setting and we seek a min-max type principle, we need an appropriate
notion of measure, such as the genus of Krasnoselskii. For the sake of
completeness we include it here (see Juutinen-Lindqvist \cite{Ju-Li-05}):

\begin{definition}
Let $E$ be a real Banach space and let $A\subset E$ be any closed symmetric
set (that is, $v\in A\Rightarrow-v\in A$). The \textit{{genus} $\gamma(A)$ of
$A$ is the smallest integer $m$ such that there exists a continuous odd
mapping $\phi:A\rightarrow\mathbb{R}^{m}\setminus{0}$. If no such integer
exists we write $\gamma(A)=\infty$. }
\end{definition}

If $0\in A$ then immediately $\gamma(A)=\infty$. On the other hand if
$\gamma(A)=1$ then $A$ is non-connected.

If we restrict ourselves to $\Sigma_{k}, k=1,2,\dots$ the collection of all
symmetric compact subsets $A\subset \mathcal{A}^{+}(m)$ such that $\gamma(A)\geq k$
then, such as for the $p$-Laplacian case (see 
\cite{GP1}), for the problem with weights it is known that (see 
\cite{SW}) there exists an increasing sequence of positive eigenvalues of
\eqref{1.1.u}, converging to $\infty$, characterized by
\begin{equation}
\label{higher-eigen}\lambda_{k,p}=\inf_{A\in\Sigma_{k}}\sup_{u\in A}%
\int_{\Omega}|\nabla u|^{p}.
\end{equation}

Observe that, since $\gamma(\{u,-u\})=1$ we recover the usual definition for
$\lambda_{1,p}$. We also recall the following lemma (see \cite{S}) that
provides a way to compute the genus of some specific subsets of $W_{0}^{1,p}$.

\begin{lemma}
Let $A\subset W_{0}^{1,p}(\Omega)$ and $U\subset\mathbb{R}^{k}$ be a bounded
neighborhood of $0$. If there exists an odd homeomorphism $h:A\rightarrow
\partial U$ then $\gamma(A)=k$.
\end{lemma}

Using this characterization we now proceed to prove the second theorem stated
in the introduction. Namely we establish an upper bound for the sequence of
eigenvalues.

\textit{Proof of Theorem \ref{teo2.intro}}. For simplicity we present the
proof for $k=2$, that is, for the second eigenvalue. The proof for $k>2$ follows by the same ideas.
Let $r_{2}>0$ be such
that there exist disjoint open balls $B_{1}=B(c_{1},r_{2})\subset\Omega$ and
$B_{2}=B(c_{2},r_{2})\subset\Omega$ with $c_{1},c_{2}\in\Omega_{+}$.
Using $r_{2}$ we define the truncated cone functions $C_{1},C_{2}$ by
\[
C_{1}(x):=\left(  r_{2}-|x-c_{1}|\right)  ^{+},\quad C_{2}(x):=\left(
r_{2}-|x-c_{2}|\right)  ^{+}.
\]

Set $A:=\,<C_{1},C_{2}>\cap\{v\in W_{0}^{1,\infty}:\Vert{v}\Vert
_{\infty,\Omega}=1\}$. We have that $A$ is closed and, by the previous lemma,
has genus 2. Therefore
\[
\lambda_{2,p}^{1/p}\leq\sup_{v\in A}\frac{\displaystyle \left(  \int_{\Omega}|\nabla
v|^{p}\right)  ^{1/p}}{\displaystyle  \left(  \int_{\Omega}m|v|^{p}\right)  ^{1/p}}.
\]
Now let $v:=\alpha C_{1}+\beta C_{2}$. Since $C_{1}$ and $C_{2}$ have disjoint
support we can write,
\[
\int_{\Omega}|\nabla v|^{p}=\left(  |\alpha|^{p}+|\beta|^{p}\right)
|B_{r_{2}}|.
\]

On the other hand, after a change of variables, we obtain, 
\[
\int_{\Omega}m(x)|v|^{p}=\int_{B_{r_{2}}(0)}\left(  \left\vert \alpha\right\vert
^{p}m\left(  x+c_{1}\right)  +\left\vert \beta\right\vert ^{p}m\left(
x+c_{2}\right)  \right)  \left\vert r_{2}-\left\vert x\right\vert \right\vert
^{p}.
\]
By assumption we have that $m(c_{1}),m(c_{2})>0$ and thus there exists
$\delta>0$ such that $m(x+c_{1}),m(x+c_{2})>0$ for $x\in B_{\delta}(0)$. Therefore,
\[
\begin{array}
[c]{l}
\displaystyle\left\Vert {v}\right\Vert _{L^{p}(\Omega,m)}
\displaystyle=\left(  \int_{B_{\delta}(0)}\left(  \left\vert \alpha\right\vert
^{p}m\left(  x+c_{1}\right)  +\left\vert \beta\right\vert ^{p}m\left(
x+c_{2}\right)  \right)  \left\vert r_{2}-\left\vert x\right\vert \right\vert
^{p}\right.  +\\
\displaystyle\hspace{7em} +\left.  \int_{B_{r_{2}}(0)\setminus B_{\delta}%
(0)}\left(  \left\vert \alpha\right\vert ^{p}m\left(  x+c_{1}\right)
+\left\vert \beta\right\vert ^{p}m\left(  x+c_{2}\right)  \right)  \left\vert
r_{2}-\left\vert x\right\vert \right\vert ^{p}\right)  ^{1/p}\\
\displaystyle\geq\left(  \int_{B_{\delta/2}(0)}\left\vert \left(  |\alpha
|^{p}m(x+c_{1})+|\beta|^{p}m(x+c_{2})\right)  ^{1/p}(r_{2}-\delta
/2)\right\vert ^{p}\right.  -\\
\displaystyle\hspace{7em}-\left.  \left(  |\alpha|^{p}+|\beta|^{p}\right)
\Vert m\Vert_{L^{\infty}}\int_{B_{r_{2}}(0)\setminus B_{\delta}(0)}%
|r_{2}-\delta|^{p}\right)  ^{1/p}\\
\qquad\rightarrow r_{2}-\delta/2\qquad\mbox{ as }p\rightarrow\infty.
\end{array}
\]
Similarly,
\[
\begin{array}
[c]{l}
\displaystyle\left\Vert {v}\right\Vert _{L^{p}(\Omega,m)}
\displaystyle=\left(  \int_{B_{\delta}(0)}\left(  |\alpha|^{p}m(x+c_{1}
)+|\beta|^{p}m(x+c_{2})\right)  |r_{2}-|x||^{p}\right.  +\\
\displaystyle\hspace{7em}\left. +  \int_{B_{r_{2}}(0)\setminus B_{\delta}
(0)}\left(  |\alpha|^{p}m(x+c_{1})+|\beta|^{p}m(x+c_{2})\right)
|r_{2}-|x||^{p}\right)  ^{1/p}\\
\displaystyle\leq\left(  \int_{B_{\delta/2}(0)}|\left(  |\alpha|^{p}
m(x+c_{1})+|\beta|^{p}m(x+c_{2})\right)  ^{1/p}r_{2}|^{p}+\right.  \\
\displaystyle\hspace{7em}+\left.  \left(  |\alpha|^{p}+|\beta|^{p}\right)
\Vert m\Vert_{L^{\infty}}\int_{B_{r_{2}}(0)\setminus B_{\delta}(0)}
|r_{2}-\delta|^{p}\right)^{1/p}\\
\qquad\rightarrow r_{2}\qquad\mbox{ as }p\rightarrow\infty.
\end{array}
\]
Since $\delta$ can be chosen arbitrarily small we conclude that $$\lim
_{p\rightarrow\infty}\left\Vert {v}\right\Vert _{L^{p}(B_{r_{2}}(0),m)}=r_{2}.
$$ 
Now, taking limits in $p$ in the inequality for the eigenvalue we have
that
\[
\limsup_{p\rightarrow\infty}\lambda_{2,p}^{1/p}\leq\frac{1}{r_{2}}.
\]
Therefore, as this inequality holds for any $r_{2}$ as above, we get
\[
\limsup_{p\rightarrow\infty}\lambda_{1,p}^{1/p}\leq\inf_{r_{2}}\frac{1}{r_{2}%
}=\frac{1}{R_{2,+}}.
\]
The proof is completed. \qed

The upper bound established above is actually attained in the case $k=2$, that
is, we can completely characterize $\lambda_{2,\infty}$ by means of $R_{2,+}$,
the maximum possible radius of two disjoint balls in $\Omega$ centered at
$\Omega_{+}$. Given the result of Theorem \ref{teo2.intro}, we only need to
prove that the reverse inequality holds, when $k=2$.

Arguing as in the proof of Theorem \ref{teo1.intro} we can easily deduce that
at least a subsequence of the sequence of eigenfunctions $\{u_{2,p}\}$
converges uniformly in $\overline{\Omega}$ to $u_{2,\infty}$. Moreover, this
function $u_{2,\infty}$ is a viscosity solution of a problem such as
\eqref{ecuacion.infty.intro} with $\lambda_{1,\infty}$ replaced by some
$\Lambda$ satisfying $\Lambda\leq\frac{1}{R_{2,+}}$.\smallskip

\textit{Proof of Theorem \ref{teo3.intro}}. From the condition imposed on all
eigenfunctions we can deduce that
\[%
\begin{array}
[c]{l}%
\displaystyle1=\left(  \int_{\Omega}m|u_{2,p}|^{p}\right)  ^{1/p}\\
\displaystyle\qquad\leq\left(  \int_{\Omega_{+}}m|u_{2,p}|^{p}\right)
^{1/p}\\
\displaystyle\qquad=\left(  \int_{\Omega_{+}}m|u_{2,p}^{+}|^{p}+\int
_{\Omega_{+}}m|u_{2,p}^{-}|^{p}\right)  ^{1/p}\\
\qquad\rightarrow\max\left\{  \Vert u_{2,\infty}^{+}\Vert_{L^{\infty}%
(\Omega_{+})},\Vert u_{2,\infty}^{-}\Vert_{L^{\infty}(\Omega_{+})}\right\}
\qquad\mbox{ as }p\rightarrow\infty,
\end{array}
\]
where $u_{2,\infty}^{\pm}$ denote the positive and negative parts of
$u_{2,\infty}$.

Since $u_{2,\infty}^{\pm}$ are Lipschitz continuous in $\overline{\Omega}$
there exist $x_{1},x_{2}\in\Omega_{+}$ such that
\[
\Vert u_{2,\infty}^{+}\Vert_{L^{\infty}(\Omega_{+})}=u_{2,\infty}^{+}%
(x_{1})\text{\quad and}\quad\Vert u_{2,\infty}^{-}\Vert_{L^{\infty}(\Omega
_{+})}=-u_{2,\infty}^{-}(x_{2}).
\]
Let now $\mathcal{N}^{\pm}\subset\Omega_{+}$ be nodal sets of $u_{2,\infty
}^{\pm}$ respectively and such that $x_{1}\in\mathcal{N}^{+}$ and $x_{2}%
\in\mathcal{N}^{-}$.
%Suppose that $|u_{2,\infty}^+(x_1)|\geq |u_{2,\infty}^-(x_2)|$ and let $y\in\partial\mathcal{N}^+$ such that
%$|x_1 - y | = dist (x_1, \partial \mathcal{N}^+)$. We have that
%$$
%1 \leq u_{2,\infty}^+ (x_1) =  u_{2,\infty}^+ (x_1) - u_{2,\infty}^+ (y) %\leq
%\left\Vert {\nabla u_{2,\infty }^+}\right\Vert_{\infty }|x_1-y|.
%$$

By Theorem $3.2$ in \cite{cuesta} (see also Theorem $8.1$ in \cite{Ju-Li-05}
for the classical setting), since $u_{2,\infty}^{+}>0$ is a viscosity solution
to $\min\{|\nabla v|-\Lambda v,-\Delta_{\infty}v\}=0$ in $\mathcal{N}^{+}$ and
$u_{2,\infty}^{+}=0$ on $\partial\mathcal{N}^{+}$, then
\[
\Lambda=\lambda_{1,\infty}(\mathcal{N}^{+})=\Vert\nabla u_{2,\infty}^{+}%
\Vert_{\infty,\mathcal{N}^{+}}=\frac{1}{R_{1,+}(\mathcal{N}^{+})}.
\]
Moreover $u_{2,\infty}^{+}(x_{1})=1.$ We obtain a similar result for
$u_{2,\infty}^{-}$. As a first conclusion we see that
\[
\Vert u_{2,\infty}^{+}\Vert_{L^{\infty}(\Omega_{+})}=\Vert u_{2,\infty}%
^{-}\Vert_{L^{\infty}(\Omega_{+})}=1.
\]
Here we are using that the only (positive) eigenvalue of \eqref{eq.p} that has an associated eigenfunction of constant sign is the first eigenvalue, see e.g. \cite[Section 1 ]{cuesta2}.

On the other hand, we also conclude that there exist two balls $B_{+}
\subset\mathcal{N}^{+}$ and $B_{-}\subset\mathcal{N}^{-}$ with radius
$R_{1,+}(\mathcal{N}^{\pm})$ respectively. Since $B_{+}$ and $B_{-}$ are
disjoint and both contained in $\Omega_{+}$, by the definition of $R_{2,+}$ we
have that
\[
R_{2,+}\geq\max\{R_{1,+}(\mathcal{N}^{+}),R_{1,+}(\mathcal{N}^{-})\}.
\]

Finally,
\[
\frac{1}{R_{2,+}}\leq\frac{1}{R_{1,+}}=\Vert\nabla u_{2,\infty}^{\pm}
\Vert_{\infty,\mathcal{N}^{\pm}}\leq\Vert\nabla u_{2,\infty}\Vert_{\infty}
\leq\liminf_{p\rightarrow\infty}\lambda_{2,p}^{1/p}\leq\frac{1}{R_{2,+}},
\]
where we have used the result of the previous theorem in the last inequality. Hence,
\[
\lambda_{2,\infty}=\frac{1}{R_{2,+}}.
\]
This ends the proof. \qed

\subsection{Examples} Now let us present some simple examples to see how the geometry 
of $\Omega_+$ affects the eigenvalues $\lambda_{1,\infty} (m)$ and $\lambda_{2,\infty} (m)$.
Notice that the size of the weight is not relevant for the limit eigenvalue problem, what matters 
here is the set $\Omega_+=\overline{\{m>0\}}$.

In what follows we will fix $\Omega$ as being the unit ball in $\mathbb{R}^2$ (for simplicity).
In this case we have
$$
\lambda_{1,\infty} (1) = 1, \qquad \mbox{ and } \qquad \lambda_{2,\infty} (1) = 2,
$$
see \cite{Ju-Li-05,JLM}.

{\bf Example 1.} Let $\Omega_+ = \overline{B_\delta (0)}$ with $\delta$ small
be a ball centered at the origin. From our results we obtain
$$
\lambda_{1,\infty} (m) = 1, \qquad \mbox{ and } \qquad \lambda_{2,\infty} (m) = \frac{1}{\delta}.
$$

{\bf Example 2.} Let $\Omega_+ = \{ x\in B_1(0) : \mbox{dist} (x ,\partial B_1(0)) \leq \delta \} $ be a small
strip around the boundary $\partial B_1(0)$ of width $\delta$.
Now, we have
$$
\lambda_{1,\infty} (m) = \frac{1}{\delta}, \qquad \mbox{ and } \qquad \lambda_{2,\infty} (m) = \frac{1}{\delta}.
$$
Notice that in this case we have $\lambda_{1,\infty} (m) = \lambda_{2,\infty} (m)$.

{\bf Example 3.} Let $\Omega_+ = \overline{B_\delta ((1/2,0))} \cup \overline{B_\delta ((-1/2,0))}$
the union of two small balls. In this case we get
$$
\lambda_{1,\infty} (m) = \frac{2}{1 + 2 \delta}, \qquad \mbox{ and } \qquad \lambda_{2,\infty} (m) = 2.
$$

\section{The first eigenvalue for a slightly different operator}
\label{sect.Other}

In this section we analyze a slightly different operator, namely we now
investigate the following eigenvalue problem
\begin{equation}
\left\{
\begin{array}
[c]{ll}
-\Delta_{p}u(x)+C(x)|u(x)|^{p-2}u(x)=\lambda m(x)|u|^{p-2}u(x) & x\in\Omega,\\
u(x)=0 & x\in\partial\Omega,
\end{array}
\right.  \label{eq.p.other}
\end{equation}
where $C$ is continuous and positive in $\overline{\Omega}$ and 
$m$ changes sign and satisfies the conditions imposed in the previous sections.

It is known (see \cite{cuesta2}) that there exists a principal eigenvalue
\begin{equation}
\lambda_{1,p}(C,m)=\min_{\int_{\Omega}m|u|^{p}=1}\int_{\Omega}|\nabla
u|^{p}+C|u|^{p}. \label{eigen.C}
\end{equation}

Our aim is to compute the limit 
$$
\lim_{p\to \infty} (\lambda_{1,p}(C,m))^{1/p} .
$$

\begin{proof}[Proof of Theorem \ref{teo.C.intro}]
Following the ideas of Theorem \ref{teo1.intro} we search for an upper bound
for $\limsup_{p}\lambda_{1,p}^{1/p}$. To this end, let $c\in\Omega_{+}$. 
Associated with this $c\in\Omega_{+}$, let $R = R(c)>0$ be the radius of
the biggest ball $B_{R}(c)$ such that $B=B_{R}
(c)\subset\Omega$. Once again we consider the function
\[
w:=\left\{
\begin{array}
[c]{ll}
d(x,\partial B_{R}(c)) & \text{if }x\in B_{R}(c),\\
0 & \text{if }x\not \in B_{R}(c).
\end{array}
\right.
\]
Using the definition of $\lambda_{1,p}$ we see that
\[
\lambda_{1,p}^{1/p}(C,m)\leq\frac{ \displaystyle\left(  |B|+\int_{B_{R}} C|w|^{p}\right)
^{1/p}}{ \displaystyle \left(  \int_{\Omega}m|w|^{p}\right)  ^{1/p}}.
\]
We already know that $$\displaystyle{\lim_{p\to\infty}\left(  \int_{\Omega
}m|w|^{p}\right)  ^{1/p}=R}.$$ On the other hand, as $C$ is positive, we obtain
$$
\left( \int_{B} C|w|^{p} \right)^{1/p}= \left( \int_{B}
C|\left(  R-d(x,c)\right)^{+}|^{p} \right)^{1/p}
\to R,
$$
as $p\to \infty$.
Therefore, letting $p$ to infinity, we obtain
\begin{equation}
\label{eigen.C.lim}\limsup_{p\to\infty}( \lambda_{1,p}(C,m))^{1/p}\leq
\max\left\{  \frac{1}{R}, 1\right\}  .
\end{equation}
We have that if $R\leq1$ then $1/R\geq1$ so that
the maximum is achieved for $1/R$. On the other hand, if
$R>1$ then the maximum is $1$.

Now, taking into account that $c\in\Omega_{+}$ we obtain 
\begin{equation}
\label{eigen.C.lim2}\limsup_{p\to\infty} (\lambda_{1,p}(C,m))^{1/p}\leq
\inf_{c\in\Omega_{+}}
\max\left\{  \frac{1}{R(c)}, 1\right\}  = \max\left\{  \frac{1}{R_+}, 1\right\},
\end{equation}
where, as before, $R_{+}:=\max
_{x\in\Omega_{+}}d(x,\partial\Omega)$.

From this bound we can argue as before to obtain that $\{u_{p}\}$ is a bounded sequence in $W_{0}^{1,q}(\Omega)$ and then
there is a subsequence (that we still call $u_{p}$) that converges weakly in
$W_{0}^{1,q}(\Omega)$ and uniformly in $\overline{\Omega}$ to a limit
$u_{\infty}$. Moreover, it holds that
\[
\left\Vert {\nabla u_{\infty}}\right\Vert _{\infty}\leq \limsup_{p\to\infty} (\lambda_{1,p}(C,m))^{1/p} \leq 
\max\left\{ \frac{1}{R_{+}},1 \right\}.
\]
Hence $u_{\infty}$ belongs to $W_{0}^{1,\infty}(\Omega)$. Moreover,
\[
\displaystyle1=\left(  \int_{\Omega}m|u_{p}|^{p}\right)  ^{1/p}\leq\left(
\int_{\Omega}m^{+}|u_{p}|^{p}\right)  ^{1/p}\rightarrow\Vert u_{\infty}%
\Vert_{L^{\infty}(\Omega_{+})}\qquad\mbox{ as }p\rightarrow\infty.
\]
Therefore, $\Vert u_{\infty}\Vert_{L^{\infty}(\Omega_{+})}\geq1$. Next we
notice that (since $u_{\infty}$ is Lipschitz continuous in $\overline{\Omega}%
$) there exists $x_{0}\in\Omega_{+}$ with
\[
u_{\infty}(x_{0})=\Vert u_{\infty}\Vert_{L^{\infty}(\Omega_{+})}\geq1.
\]
Now we observe that, if we take $y\in\partial\Omega$ such that $|x_{0}-y|=$
dist$(x_{0},\partial\Omega)$, we have
\[
1\leq u_{\infty}(x_{0})-u_{\infty}(y)\leq\left\Vert {\nabla
u_{\infty}}\right\Vert _{\infty}|x_{0}-y|\leq \max\left\{ \frac{1}{R_{+}},1 \right\} |x_{0}-y|
\leq \max\left\{ 1,R_{+} \right\}.
\]
Hence, if $R_{+}:=\max_{x\in\Omega_{+}}d(x,\partial\Omega)\leq 1$ we get that all
the previous inequalities must be equalities and so
\[
u_{\infty}(x_{0})=1,\qquad\left\Vert {\nabla u_{\infty}}\right\Vert _{\infty
}=\frac{1}{R_{+}},\qquad\mbox{and}\qquad d(x_{0},\partial\Omega)=R_{+}.
\]
Notice that this implies that $u_{\infty}$ is a minimizer for the limit
variational problem, that is,
\[
\frac{ \max\{ \left\Vert \nabla u_{\infty}\right\Vert _{L^{\infty}(\Omega)};
\left\Vert u_{\infty}\right\Vert _{L^{\infty}(\Omega)}\} }
{\left\Vert u_{\infty}\right\Vert _{L^{\infty}(\Omega^{+})}}=\inf_{v\in
W_{0}^{1,\infty}(\Omega)} \frac{\max\{ \left\Vert \nabla v \right\Vert _{L^{\infty}(\Omega)};
\left\Vert v \right\Vert _{L^{\infty}(\Omega)}\} }{\left\Vert v\right\Vert _{L^{\infty}(\Omega^{+})}} = \frac{1}{R_{+}}.
\]
Moreover, we have
$$
\lim_{p\to \infty}  (\lambda_{1,p}(C,m))^{1/p} =  \frac{1}{R_{+}} = \max\left\{ \frac{1}{R_{+}},1 \right\}.
$$

On the other hand, if $R_{+}:=\max_{x\in\Omega_{+}}d(x,\partial\Omega)\geq 1$  we have
$$
\liminf_{p\to \infty} \left(\int_{\Omega}C|u_p|^{p} \right)^{1/p} \geq u_\infty (x_0) = 1.
$$
Therefore,
$$
\begin{array}{l}
\displaystyle 
\liminf_{p\to \infty}  (\lambda_{1,p}(C,m))^{1/p}= \liminf_{p\to \infty}  \left(\min_{\int_{\Omega}m|u|^{p}=1}\int_{\Omega}|\nabla
u|^{p}+C|u|^{p} \right)^{1/p}  \\
\qquad \displaystyle \geq  \liminf_{p\to \infty}  \left(\int_{\Omega}C|u_p|^{p} \right)^{1/p} \geq 1.
\end{array}
$$
We conclude that also in this case
$$
\lim_{p\to \infty}  (\lambda_{1,p}(C,m))^{1/p} =1 = \max\left\{ \frac{1}{R_{+}},1 \right\}.
$$
This ends the proof.
\end{proof}

\end{document}